 \documentclass[11pt,letterpaper]{amsart}

 \usepackage{txfonts} 

 \usepackage[colorlinks,linkcolor=blue,citecolor=blue]{hyperref}
 
 \theoremstyle{plain}

\newtheorem{theorem}{Theorem}[section]
\newtheorem{lemma}[theorem]{Lemma}

\theoremstyle{definition}

\theoremstyle{remark}

\numberwithin{equation}{section}

\begin{document}

\title[]
{Notes on conformal metrics of negative curvature on manifolds with boundary }


\author{Rirong Yuan }
\address{School of Mathematics, South China University of Technology, Guangzhou 510641, China}
\email{yuanrr@scut.edu.cn}



\dedicatory{}

\begin{abstract}
	
We use certain Morse functions to construct  conformal metrics 
such that the eigenvalue vector of modified Schouten tensor 
 belongs to a given cone. 	As a  result, we prove that any Riemannian metric on compact 3-manifolds with boundary 
	is conformal   to a compact metric of negative sectional curvature.
	
\end{abstract}

\maketitle


  
  \section{Introduction}

  
  In Riemannian geometry, 
  a basic problem is to find a metric so that the various curvatures satisfy prescribed properties. A well-known result  on this direction is 
   the existence of metrics with negative Ricci curvature.
  In \cite{Gao1986Yau} Gao-Yau proved that any closed 3-manifold admits a Riemannian metric with negative Ricci curvature. Subsequently, Gao-Yau's theorem was  generalized
   by Lohkamp
  \cite{Lohkamp-1} to higher dimensions. 
   In addition, Lohkamp's result allows  manifolds with boundary.

 	This note is devoted to constructing conformal metrics with certain restrictions to curvatures.
 	 	Let  $K_g$, ${Ric}_g$ and   ${R}_g$ denote   sectional,  Ricci and scalar curvature of the Riemannian metric $g$, respectively, 
 	with respect to the  Levi-Civita connection. 
 	
 	 		Our first result is concerned with sectional curvature  of   3-manifolds. 
 	 		\begin{theorem}
 	 			\label{thm0-main-dimension3}
 	 			Let $(M,g)$ be a three dimensional compact Riemannian manifold with smooth boundary. Then there is a smooth compact conformal metric $g_u=e^{2u}g$ of negative sectional curvature.
 	 			
 	 		\end{theorem}

 	 The proof is based on Morse theory as well as  the relation between sectional curvature and Einstein tensor in dimension three.  
 	 
 	 \begin{lemma}
 	 	\label{prop1-einstein-sectional}
 	 	Fix $x\in M^3$,  let $\Sigma\subset T_xM$ be a tangent $2$-plane,   $\vec{\bf n}\in T_xM$ the unit normal vector to $\Sigma$, then  
 	 	\begin{equation} \label{sectional-einstein} \begin{aligned}
 	 			G_g(\vec{\bf n},\vec{\bf n})=-K_g(\Sigma). \nonumber
 	 	\end{aligned}  \end{equation} 
 	 	Here  
 	 	$G_g:=Ric_g-\frac{R_g}{2} g$ stands for the Einstein tensor of $g$.
 	 \end{lemma}
  
   	 		Despite the Cartan-Hadamard theorem and the complex of topology of underlining manifolds, in general one could not except that the resulting metric in Theorem \ref{thm0-main-dimension3} is complete.  
   	 		A nice complement is due to Gursky-Streets-Warren  \cite{Gursky-Streets-Warren2010}  who proved that any Riemannian metric on compact 3-manifold  with smooth boundary admits a complete conformal metric of ``almost negative"  sectional	curvature.

 	 In fact we prove more general results than Theorem \ref{thm0-main-dimension3}.
	Let $\Gamma$ be an \textit{open},  \textit{symmetric}, \textit{convex} cone in $\mathbb{R}^n$ with  vertex at the origin, $\partial \Gamma\neq \emptyset$, and
 	$$\Gamma_n:=\left\{\lambda=(\lambda_1,\cdots,\lambda_n)\in \mathbb{R}^n: \mbox{ each } \lambda_i>0\right\}\subseteq\Gamma.$$
 	Following \cite{CNS3}   $\Gamma$ is a type 1 cone if $(0,\cdots,0,1)\in\partial\Gamma$; otherwise $\Gamma$ is a type 2 cone. As in \cite{yuan-PUE2-note} for $\Gamma$ we denote $\varrho_\Gamma$  the constant 
 	with
 	\begin{equation}
 		\label{varrho1-Gamma}
 		\begin{aligned}
 			(1,\cdots,1,1-\varrho_\Gamma)\in\partial \Gamma.
 		\end{aligned}
 	\end{equation}  

	For $n\geq3$ we 
 	denote the modified Schouten tensor by   (see \cite{Gursky2003Viaclovsky})
 	\begin{equation}  	\begin{aligned} 		A_{{g}}^{\tau}=\frac{1}{n-2} \left({Ric}_{g}-\frac{\tau}{2(n-1)}   {R}_{g}\cdot {g}\right), \,\, \tau \in \mathbb{R}. \nonumber	\end{aligned} \end{equation}  
 For $\tau=n-1$,   it corresponds to the Einstein tensor.    
 	When $\tau=1$, it is the Schouten tensor 
 	$$A_g=\frac{1}{n-2} \left({Ric}_g-\frac{1}{2(n-1)}{R_g}\cdot g\right).$$
 	For simplicity, we denote $A_{{g}}^{\tau,\alpha}=\alpha A_{{g}}^{\tau}$ with $\alpha=\pm1$.
Also  we  denote the boundary of $M$ by  $\partial M$. Let $\bar M=M\cup\partial M$. 


In this note we consider the problem of finding   a conformal metric $g_u$ subject to
 \begin{equation}
 	\label{def1-admissible-metric}
 	\begin{aligned}
 		\lambda( g^{-1} A_{g_u}^{\tau,\alpha}) \in\Gamma.  \nonumber
 	\end{aligned}
 \end{equation}
The existence of such metrics 
is imposed as  a key assumption to  study conformal deformation for  Schouten tensor and more general modified Schouten tensors; see e.g. \cite{Guan2008IMRN,Gursky2003Viaclovsky,Li2011Sheng}. 
In
prequels \cite{yuan-PUE1,yuan-PUE2-note} the author constructed
 the conformal metric with $\lambda( -g^{-1} A_{g_u}^{\tau}) \in\Gamma$ for $\tau<1$ and $\tau\leq 2-\frac{2}{\varrho_\Gamma}$, and then extended some of results in Lohkamp \cite{Lohkamp-1,Lohkamp-2} 
  using fully nonlinear elliptic equations.
 The case $\tau>1$ is  different.
When $\tau\geq 2$ the  metrics with $\lambda( g^{-1} A_{g_u}^{\tau}) \in\Gamma$ were also constructed in the same papers 
 under the extra assumption
\begin{equation}
	\label{assumption-4-2}
	\begin{aligned}
		(1,\cdots,1,1-	\frac{n-2}{\tau-1} )\in \Gamma,
	\end{aligned}
\end{equation} 
i.e., $\tau>1+(n-2)\varrho_{\Gamma}^{-1}$.
%
Nevertheless,  \eqref{assumption-4-2} does not allow $\tau=n-1$ when  $\Gamma=\Gamma_n$.
 This assumption is dropped in the following theorem. 
  \begin{theorem}	\label{thm2-main}
  	Let $(M,g)$ be a compact Riemannian manifold of dimension $n\geq 3$ and with smooth boundary. 
  	Then for any $\tau\geq2$ there is a smooth compact  metric $g_u=e^{2u}g$ for some $u\in C^\infty(\bar M)$ with 
  	$A^{\tau}_{g_u}>0 \mbox{ in } \bar M.$
  \end{theorem}



 
 As a consequence,  we obtain Theorem \ref{thm0-main-dimension3}.
 The   case $\tau=1$ is   more delicate. 

 \begin{theorem}
 	Let $(M,g)$ be a compact Riemannian manifold of dimension $n\geq 3$ and with smooth boundary. There is a smooth compact   metric $g_u=e^{2u}g$, $u\in C^\infty(\bar M)$ such that
 	\begin{equation}
 		\begin{aligned}
 			\lambda(-g^{-1}A_{g_u}) \in \Gamma \mbox{ in } \bar M, \nonumber
 		\end{aligned}
 	\end{equation}
 	provided that $\varrho_\Gamma>2$. 
 	
 	If, in addition, $\Gamma$ is of type 2  then the statement holds for $\varrho_\Gamma\geq 2.$
 \end{theorem}
 
 
  	 



\medskip

\section{Preliminaries}
\label{sec2-prelimi}

\subsection{Notations and Formulas}


Let $e_1,...,e_n$ be a local frame on $M$.  
Denote
$$  \langle X,Y\rangle=g(X,Y),\,\  g_{ij}= \langle e_i,e_j\rangle,\,\   \{g^{ij} \} =  \{g_{ij} \}^{-1}.$$
Under Levi-Civita connection $\nabla$ of $(M,g)$, $\nabla_{e_i}e_j=\Gamma_{ij}^k e_k$, and $\Gamma_{ij}^k$ denote the Christoffel symbols.  
For simplicity we write 
$$\nabla_i=\nabla_{e_i}, \nabla_{ij}=\nabla_i\nabla_j-\Gamma_{ij}^k\nabla_k, 
\nabla_{ijk}=\nabla_i\nabla_{jk}-\Gamma_{ij}^l\nabla_{lk}-\Gamma^l_{ik}\nabla_{jl}, \mbox{ etc}.$$

Under the conformal change
${g}_u=e^{2u}g$, one has  
(see e.g. \cite{Besse1987} or \cite{Gursky2003Viaclovsky}) 
\begin{equation}
	\label{conformal-formula1}
	\begin{aligned}
		A_{\tilde{g}}^{\tau}
		= A_{g}^{\tau}
		+\frac{\tau-1}{n-2}\Delta u g- \nabla^2 u
		+\frac{ \tau-2}{2}|\nabla u|^2 g 
		+   du\otimes du. 
	\end{aligned}
\end{equation}
In particular, the Schouten tensor obeys
\begin{equation}\label{conformal-formula2}\begin{aligned}A_{{g}_u} = A_{g} -\nabla^2 u	-\frac{1}{2}|\nabla u|^2 g+ du\otimes du. \end{aligned}\end{equation}
Throughout this paper, $\Delta u$, $\nabla^2 u$ and $\nabla u$ are the Laplacian, Hessian and gradient of $u$ with respect to $g$, respectively.
For simplicity,  we denote
\begin{equation}
	\label{beta-gamma-A2}
	\begin{aligned}
		V[u]=\Delta u g -\varrho\nabla^2 u+\gamma |\nabla u|^2 g +\varrho du\otimes du+A,   
	\end{aligned} 
\end{equation} 
 Throughout this paper,  we denote
\begin{equation}
	\label{beta-gamma-A-3}
	\begin{aligned}
		\varrho=\frac{n-2}{\tau-1}, \mbox{ }
		\gamma=\frac{(\tau-2)(n-2)}{2(\tau-1)}.
	\end{aligned}
\end{equation}

One can check 
 $$V[u]=\frac{n-2}{\tau-1} A^{\tau}_{{g}_u}, \mbox{ } {g}_u=e^{2u}g, \mbox{ } A=\frac{n-2}{\tau-1} A_{g}^{\tau}.$$

\subsection{Some result on Morse function}

The following lemma asserts that any compact manifold with boundary carries a Morse function without any critical point.    

\begin{lemma}
	\label{lemma-diff-topologuy}
	Let $M$
	be a compact connected  
	manifold of dimension $n\geq 2$ with smooth boundary. Then there is a smooth function $v$ without any critical points, that is $d v\neq 0$.
\end{lemma}

\begin{proof} 
	     The construction  is more or less standard in differential topology.
	Let $X$ be the double of $M$. Let $w$ be a smooth Morse function on $X$ with the critical set $\{p_i\}_{i=1}^{m+k}$, among which $p_1,\cdots, p_m$ are all the critical points  being in $\bar M$. 
	Pick $q_1, \cdots, q_m\in X\setminus \bar M$ but not the critical point of $w$. By homogeneity lemma 
	(see \cite{Milnor-1997}), 
	one can find a diffeomorphism
	$h: X\to X$, which is smoothly isotopic to the identity, such that 
	\begin{itemize}
		\item $h(p_i)=q_i$, $1\leq  i\leq  m$.
		\item $h(p_i)=p_i$, $m+1\leq  i\leq  m+k$.
	\end{itemize}
	Then $v=w\circ h^{-1}\big|_{\bar M}$ is the desired 
	function.
	
\end{proof}


\section{Construction of  metrics}


Let $v$ be the smooth function as we claimed in Lemma \ref{lemma-diff-topologuy}. 
Without loss of generality, we assume 
\begin{equation}
	\begin{aligned}
		v\geq 1 \mbox{ in } \bar M. \nonumber
	\end{aligned}
\end{equation}
Let us
take ${u}=e^{Nv},$ ${g}_u=e^{2{u}}g,$
then 
\begin{equation}
	\begin{aligned}
		-A_{g_u}
		=  -A_g +Ne^{Nv} \nabla^2 v 
		+N^2 e^{Nv} dv\otimes dv 
		+ 
		N^2 e^{2Nv} \left(\frac{1}{2}|\nabla v|^2 g -dv\otimes dv \right), 	\nonumber
	\end{aligned}
\end{equation}  
\begin{equation}
	\label{key3-2}
	\begin{aligned}
		V[{u}]
		=N^2e^{Nv}\left( (1+\gamma e^{Nv})|\nabla v|^2 g+\varrho (e^{Nv}-1)dv\otimes dv\right) 
		+Ne^{Nv} (\Delta v g-\varrho \nabla^2 v)+A 	\nonumber
	\end{aligned}
\end{equation}	
and
\begin{equation}
	\label{key3}
	\begin{aligned}
		\,&\lambda(g^{-1} ((1+\gamma e^{Nv})|\nabla v|^2 g+\varrho (e^{Nv}-1)dv\otimes dv ))\\
		=\,& |\nabla v|^2  \left[(1, \cdots, 1, 1-\varrho)+  e^{Nv}  (\gamma,\cdots,\gamma,\gamma+\varrho)\right]. 	
	\end{aligned}
\end{equation}

 \subsection{\bf The case $\tau\geq2$.}
 In this case $$\gamma\geq 0,  \,\, \varrho>0.$$ One easily knows that
\begin{equation}
	\begin{aligned}
	\,&	(1, \cdots, 1, 1-\varrho)+  e^{Nv}  (\gamma,\cdots,\gamma,\gamma+\varrho) \\
		=\,& (1+\gamma e^{Nv},\cdots,1+\gamma e^{Nv},1+\gamma e^{Nv}+(e^{Nv}-1)\varrho).
		\nonumber
	\end{aligned}
\end{equation}

Let us take $N\gg1$ we know that
\begin{equation}
	\begin{aligned}
		V[u]>0 \mbox{ in } \bar M. 	\nonumber
	\end{aligned}
\end{equation}

 \subsection{\bf The case $\tau=1$.} It is easy to check that 
\[\lambda\left(g^{-1}(\frac{1}{2}|\nabla v|^2 g -dv\otimes dv) \right)=|\nabla v|^2\left(\frac{1}{2},\cdots,\frac{1}{2},-\frac{1}{2} \right).\]

When $\varrho_\Gamma>2$,  
\[\lambda\left(g^{-1}(\frac{1}{2}|\nabla v|^2 g -dv\otimes dv) \right)\in\Gamma \mbox{ in } \bar M.\]
Therefore, when $N\gg1$ we see \[\lambda(-g^{-1}A_{g_u})\in\Gamma \mbox{ in } \bar M.\] 

If, in addition, $\Gamma$ is of type 2, then $(0,\cdots,0,1)\in\Gamma$. Thus
$$\lambda(g^{-1}(-A_g +Ne^{Nv} \nabla^2 v +N^2 e^{Nv} dv\otimes dv ))\in\Gamma$$ provided $N\gg1.$ Together with $\varrho_\Gamma\geq 2$, we know $\lambda(-g^{-1}A_{g_u})\in\Gamma$ if $N\gg1.$


\bigskip


\end{document}